\documentclass{amsart}
\usepackage{amsmath,amsthm,amssymb,xypic}
\usepackage[all]{xy}
\theoremstyle{plain}
\newtheorem{Th}{Theorem}

\newtheorem{theorem}{Theorem}[section]

\newtheorem{proposition}[theorem]{Proposition}
\newtheorem{lemma}[theorem]{Lemma}
\newtheorem{corollary}[theorem]{Corollary}

\newtheorem{example}[theorem]{Example}

\theoremstyle{definition}

\newtheorem{definition}[theorem]{Definition}

\theoremstyle{remark}

\newtheorem*{claim}{Claim}

\newtheorem{case}[theorem]{Case}

\newtheorem{remark}[theorem]{Remark}

\DeclareMathOperator{\Sing}{Sing}

\newcommand{\QED}{\ifhmode\unskip\nobreak\fi\quad {\rm Q.E.D.}} 

\newcommand\iso{\cong}

\newcommand{\f}{\varphi}

\newcommand{\I}{\mathcal{I}}

\renewcommand{\L}{\mathcal{L}}

\renewcommand{\O}{\mathcal{O}}
\renewcommand{\o}{\mathcal{O}}
\renewcommand{\P}{\mathbb{P}}

\newcommand{\Q}{\mathbb{Q}}

\newcommand{\rat}{\dasharrow}

\title{Birational geometry of rational quartic surfaces}

\author{Massimiliano Mella}
\address{Dipartimento di Matematica e Informatica\\
Universit\`a di Ferrara\\
Via Machiavelli 30\\
44121 Ferrara, Italia} \email{mll@unife.it}

\date{April 2019}

\begin{document}
\maketitle
\section*{Introduction}

Let $X\subset\P^n$ be an irreducible and reduced projective variety over an algebraically
closed field. A classical question is to study the birational
embedding of $X$ in $\P^n$ under the action of  the Cremona group of $\P^n$. 
In other words considering $X_1$ and $X_2$,
two birationally equivalent
projective varieties in $\P^n$,   one wants to understand  if there exists  a
Cremona transformation of $\P^n$ that maps $X_1$ to $X_2$. If this is
the case  $X_1$ is said to be Cremona Equivalent (CE) to $X_2$, see Definition~\ref{def:CE}. This
projective statement  can also be interpreted in terms of log Sarkisov
theory, \cite{BM}, and is related to the  Abhyankar--Moh problem, \cite{AM} and
\cite{Je}. In the latter paper it is proved, using techniques derived
form  Abhyankar--Moh problem, that over the complex
field the birational
embedding is unique as long as $\dim X< \frac{n}2$. 
The general problem is then completely solved in \cite{MP} where it is proved
that this is the case as long as the codimension of $X_i$ is at
least 2, see also \cite{CCMRZ} for a more algorithmic way to produce
the required Cremona equivalence. Examples of inequivalent embeddings of divisors are well
known, see also \cite{MP}, in all dimensions.
The problem of Cremona equivalence is therefore reduced to study the
action of the Cremona group on  divisors.   A class of divisors for which a reasonable answer is
known is that of cones. In  \cite{Me2} it is proved that two cones are CE if their
hyperplane sections are birational.

The special case of plane curves has been widely treated  both in
the old times, \cite{Co}, \cite{SR}, \cite{Ju}, and more recently, \cite{Na}, \cite{Ii}, \cite{KM},
\cite{CC}, and \cite{MP2}, see also \cite{BB} for
a nice survey. In \cite{CC} and \cite{MP2} a complete
description of plane curves up to Cremona equivalence is given and in
\cite{CC} a detailed study of the Cremona equivalence for linear systems 
is furnished. 
In particular it is interesting to note that the Cremona equivalence
of a plane curve is dictated by its singularities and cannot be
guessed without a partial resolution of those, \cite[Example
3.18]{MP2}. Due to this it is quite hard even in the plane curve case to
determine the Cremona equivalence class of a fixed curve simply by its
equation.

It is then natural to investigate surfaces in $\P^3$. In this set up using the
$\sharp$-Minimal Model Program, developed in \cite{Me} or minimal model program
with scaling \cite{BCHMc}, a criterion for detecting
surfaces Cremona equivalent to a plane has been given in \cite{MP2}. The criterion, inspired by
the previous work of Coolidge on curves Cremona equivalent to lines
\cite{Co}, allows to determine all rational surfaces
that are Cremona equivalent to a plane, \cite[Theorem 4.15]{MP2}. 
Unfortunately, worse than in the plane
curve case, the criterion requires not only the
resolution of singularities but also a control on different log
varieties attached to the pair $(\P^3,S)$.  So far it has
been 
impossible to apply it to explicit examples of surfaces in $\P^3$. 
The main difficulty coming from the necessity to check the threshold,
 see Definition~\ref{def:threshold}, on all
good models, see Definition~\ref{def:good_models}, of the pair
$(\P^3,S)$. In this note I start removing this condition from the
criterion in Corollary~\ref{cor:rhominorediuno}. This improvement
gives back a result that can be applied  in a wide range of
cases. Next I concentrate on the case of quartic rational
surfaces. The reason I study this special class of surfaces is twofold. Firstly it
is quite easy to study the CE up to cubic surfaces, see
Proposition~\ref{rem:cubics}. Then rational quartic surfaces are the first non
immediate case having hundreds of non isomorphic families, \cite{De} \cite{Je16}.
Beside this, there is an intrinsic  complexity of quartic
surfaces in $\P^3$ that deserves to be studied under any
perspective. Smooth quartic surfaces are the only smooth hypersurfaces
with automorphisms not coming from linear automorphism of $\P^n$,
\cite{MM}. In a recent paper K. Oguiso produced examples of isomorphic
smooth
quartic surfaces that are not CE, \cite{Og}. It is a long standing problem
to determine which quartic surfaces are stabilized by subgroup of the Cremona
group, that is for which quartic surface $S\subset\P^3$ there is a
Cremona modification $\omega:\P^3\rat\P^3$ such that $\omega$ is not
an isomorphism and $\omega(S)=S$. The above problem has been studied
by Enriques \cite{En} and Fano \cite{Fa} and also by Sharpe and coauthors in a series of papers,
\cite{MS} and \cite{SS}, at the beginning of the $XX^{\rm th}$
century. More recently Araujo-Corti-Massarenti continued the study of mildly
singular quartic surfaces admitting a non trivial  stabilizers in the Cremona
Group, \cite{ACM}, in the context of Calabi-Yau pairs preserving
symplectic forms.
In the light of these specialities of quartic surfaces the main theorem I
prove is the following, quite surprising, result.
\begin{Th} Let $S\subset\P^3$ be an irreducible and reduced rational
  quartic surface. Then $S$ is Cremona Equivalent to a plane.
\end{Th}
This shows that any rational quartic has a huge
stabilizer in the Cremona group disregarding the type of singularity it may
have. Indeed it is amazing that, even if there are hundreds of non
isomorphic families of
rational quartics, see \cite{Je16} \cite{De},
the Cremona group of $\P^3$ is playable enough to smooth any of them
to a plane. A similar statement is not true for rational surfaces of
degree at least 8, as a straightforward consequence of
Noether-Fano inequalities. I have not a precise feeling on what happens in
the remaining degrees 5,6,7, but I think it is worthwhile to study
them all.

The proof of the theorem is based on the simplified version of the
criterion in \cite{MP2} together with the analysis of some special
Cremona modification associated to linear systems of quadrics. Indeed in many
instances  it is useful to  produce a
linear system of quadrics having multiplicity half the multiplicity of
$S$ along some valuation embedded in $\Sing(S)$. Via this linear
system the quartic is often simplified and can be linearized in an easier
way.

I want to thank Ciro Ciliberto for reviving my interest in Cremona
Equivalence for rational surfaces during a very pleasant stay in
Cetraro and Igor Dolgachev for pointing out Jessop's book \cite{Je16} and
the Cyclides treated in the final Example. Thanks are due to the
referee, whose careful reading prevented misprints and clarified
the exposition.

\section{Preliminaries}
I work over the complex field.

\begin{definition}\label{def:CE}
  Let $X,Y\subset\P^N$ be irreducible and reduced subvarieties of
  dimension $r$. I say
  that $X$ is Cremona Equivalent (CE) to $Y$ if there is a
  birational modification $\phi:\P^N\rat\P^N$ such that $\phi(X)=Y$
  and $\phi$ is well defined on the generic point of $X$.
\end{definition}

It is clear that if $X$ is CE to $Y$ then $X$ and $Y$ are
birational. This necessary condition is also sufficient as long as $X$
is not a divisor by the main theorem in \cite{MP}, see
also\cite{CCMRZ}.
In this note I am interested in studying the CE of rational surfaces
of $\P^3$. For this reason   I start with some definition and
results about uniruled 3-folds.
\begin{definition}\label{def:threshold}
Let $(T,H)$ be a $\Q$-factorial uniruled
   3-fold and $H$ an irreducible and
   reduced effective Weil divisor
  on $T$. Let
  $$\rho(T,H)=:\mbox{  \rm sup   }\{m\in \Q|H+mK_T
\mbox{
   is $\Q$-linearly equivalent to an effective $\Q$-divisor  }\}$$
 be the (effective) threshold of the pair
 $(T,H)$.
\end{definition}
\begin{remark} The threshold is not a
birational invariant of  pairs and it
is not preserved by blowing up.
Consider
 a plane $H\subset\P^3$ and let
$Y\to \P^3$ be the blow up of a point in $H$ 
then $\rho(Y,H_Y)=0$, while
$\rho(\P^3,H)=1/4$. For future reference note that both are less than
one.
\end{remark}

In \cite{MP2}, to overcome this problem it was introduced the notion
of good models and of sup threshold $\overline{\rho}(T,S)$,
\cite{MP2}.  These combined allowed to
characterize the  Cremona Equivalence to a
plane, \cite[Theorem 4.15]{MP2}. The dark side of this
characterization  is the
impossibility to check it on explicit examples.

Here, by a simple trick, I want to simplify the statement of
\cite[Theorem 4.15]{MP2} to make it applicable in many instances.
For this purpose I start recalling the following definition.

\begin{definition}\label{def:good_models}
Let $(Y,S_Y)$ be a 3-fold pair. The pair $(Y,S_Y)$ is a birational
model of the pair $(T,S)$ if there is a birational map $\f:T\rat Y$
such that  $\f$ is well defined on the generic point of
$S$ and $\f(S)=S_Y$. 
A good model, \cite{MP2}, is a pair  $(Y,S_Y)$ with $S_Y$ smooth
and $Y$ terminal and $\Q$-factorial.
\end{definition}

\begin{remark}
 Let $(T,S)$ be a pair. To produce a good model it is enough to
 consider a log resolution of $(T,S)$. Clearly there are infinitely
 many good models for any pair and running a directed MMP one can find
 the one that is more suitable for the needs of the moment.
\end{remark}

My first aim is to show that to check the Cremona Equivalence to a
plane it is not necessary to go through all good models. In this
direction the first technical result I am proving is that, even if the
threshold is not a birational
invariant of the pair as a number, there is the following useful
property.

\begin{lemma}\label{lem:rhominorediuno} Let $(T,S)$ and $(T_1,S_1)$ be
  birational models of a pair. Assume that $(T,S)$ has canonical singularities. If
  $\rho(T,S)=a\geq 1$ then $\rho(T_1,S_1)\geq a$.
\end{lemma}
\begin{proof}
  Let $\f:T\rat T_1$ be a birational map with
  $\f(S)=S_1$. Let
$$
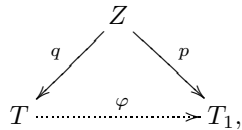
\xymatrix{
&Z\ar[dl]_{q}\ar[dr]^{p}&\\
T\ar@{.>}[rr]^{\f}&&T_1,}
$$
be a resolution of the map $\f$.

I  have 
\begin{eqnarray*}
  aK_Z+S_Z=(a-1)K_Z+K_Z+S&=p^*(a-1)K_T+\Delta+ p^*(K_T+S)+\Delta_S=\\
  &=p^*(aK_T+S_T)+\Delta+\Delta_S,
\end{eqnarray*}
for $\Q$-divisors $\Delta$ and $\Delta_S$. The pairs $(T,S)$ has
canonical singularities, therefore
$\Delta$ and $\Delta_S$ are effective divisors.
By hypothesis
$aK_T+S$ is $\Q$-effective, thus $aK_Z+S_Z$ is $\Q$-effective. Since
$$q_*(aK_Z+S_Z)\sim_{\Q}aK_{T_1}+S_1$$
this is enough to prove that $\rho(T_1,S_1)\geq a$.
\end{proof}

 As a direct consequence of Lemma~\ref{lem:rhominorediuno} I may reformulate the condition of being Cremona
 Equivalent to a plane, \cite[Theorem 4.15]{MP2}, avoiding the check
 of the threshold of all good models.

 \begin{corollary}\label{cor:rhominorediuno} A rational surface $S\subset\P^3$ is Cremona equivalent
   to a plane if and only if there is a good model $(T,S_T)$ of
   $(\P^3,S)$ with $0<\rho(T,S_T)<1$.
 \end{corollary}
 \begin{proof} By Theorem 4.15 in \cite{MP2} $S$ is Cremona
   equivalent to a plane if and only if for all  good models
   the threshold is bounded from the above by 1 and there is a good model
   with positive threshold. By  Lemma~\ref{lem:rhominorediuno}
 if there is a good model, say $(T,S_T)$, with $0<\rho(T,S_T)<1$ all
  good models have threshold bounded by 1. 
 \end{proof}

\begin{remark} Via a general projection $S\subset\P^3$ of a quintic
elliptic scroll in $\P^4$, it is
 possible to produce examples of good models $(T,S_T)$ with $T$ rational,
 $S_T$ non rational and $\rho(T,S_T)=0$. Allowing non rational
 singularities we may produce pairs with $0<\rho<1$ and $S$ non
 rational.  Let $S\subset\P^3$ be a  cone over a smooth cubic
 curve then  $\rho(\P^3,S)=3/4$ and the pair is not CE to a plane.
\end{remark}

There is a class of surfaces, and more generally hypersurfaces, that
are CE to a hyperplane.

\begin{remark}
  \label{rem:monoids} Let $X\subset\P^{n}$ be a monoid, that is an
  irreducible and reduced hypersurface of degree $d$ with a point, say $p$, of
  multiplicity $d-1$. Then I can write $X=(x_0F_{d-1}+F_d=0)$.  Let us
  consider the linear
  system 
$$\L:=\{(F_{d-1}x_1=0), \ldots , (F_{d-1}x_{n}=0), X\}. $$
Then $\f_\L:\P^n\rat\P^n$ is a birational modification and $\f_\L(X)$
is a hyperplane. It follows that any monoid is CE to a hyperplane.
\end{remark}

As a warm up  I study rational surfaces of
degree at most 3, see  \cite{MP2} and \cite{Me2}.

\begin{lemma}\label{rem:cubics}
  Let $S\subset\P^3$ be an irreducible and reduced rational surface of
  degree at most 3. Then $(\P^3,S)$ is CE to a plane. 
\end{lemma}
\begin{proof}
  The statement is immediate in degree 2 by Remark~\ref{rem:monoids}, because
  any quadric is a monoid. Let $S$ be a rational
  cubic. If $S$ is smooth then $(\P^3,S)$ is a good model with
  $\rho(\P^3,S)=3/4$, hence  I
  conclude by Corollary~\ref{cor:rhominorediuno}. If $S$ has a double
  point, then it is a monoid and I conclude again by
  Remark~\ref{rem:monoids}. If $S$ is a cone, then its plane section is
  a rational curve and I conclude by \cite[Theorem 2.5]{Me2}.
\end{proof}

\section{Rational quartic surfaces}
In this section I  study the CE of rational quartic
surfaces proving Theorem 1.
The case of  quartic surfaces singular along a twisted
cubic is by far the most interesting from a geometric point of view.
\begin{remark}
There are two classes of these surfaces,  a general projection of a rational scroll of
 degree 4 in $\P^4$ and the tangential variety of the  twisted cubic. The
 former has ordinary double points along $\Gamma$ while the latter has
 cuspidal singularities. It is interesting to note that, from the
 point of view of CE they behave in the same way.
\end{remark}
\begin{proposition}\label{exa:twistedcubic} 
  Let $S\subset\P^3$ be a quartic surface singular along a twisted
 cubic $\Gamma$. Then $S$ is CE to a plane.
\end{proposition}
\begin{proof}
Let $\nu:T\to\P^3$ be the
 blow up of $\Gamma$ with exceptional divisor $E$. Then $T$ has a
 scroll structure, say $\pi:T\to\P^2$, given by the secant lines of $\Gamma$, onto
 $\P^2$. In particular all fibers of $\pi$ are irreducible and
 reduced and $T=\P({\mathcal E})$, for a vector bundle ${\mathcal
 E}$ on $\P^2$ classically known to be defined by the following exact
sequence
$$ 0\to\o_{\P^2}(-1)^{\oplus 2}\to\o_{\P^2}^{\oplus 4}\to {\mathcal
  E}\otimes\o_{\P^2}(1)\to 0.$$
Moreover I have
 $\pi^*\o(1)=\nu^*\o(2)-E$ and  $S_T=\pi^*C$, for $C\subset\P^2$ an
 irreducible and reduced
 conic. Let $S_T=\nu^{-1}_*(S)$ be the strict stranform of the surface
 $S$. Then  I have
 $\rho(T,S_T)=0$.
 Note that at this point both the tangential variety
 and the general projection are pull backs of smooth conics. Therefore
 they are isomorphic as abstract varieties, but they differ in the scheme theoretic intersection
 with the exceptional divisor $E$. This is the only reason we end up with
 non isomorphic quartic surfaces in $\P^3$.

 Next we want to extend a standard Cremona transformation of the base
 $\P^2$ to a birational modification of $T$, following \cite[5.7.4]{Me}.

Let $f_1, f_2, f_3\subset S_T$ be three fibers of $\pi$ in general
position and
$l_i=\nu(f_i)$ the corresponding line in $\P^3$. Let $Q_i\subset\P^3$ be the
unique quadric containing $\Gamma\cup l_j\cup l_k$, with
$\{i,j,k\}=\{1,2,3\}$, and $D_i=\nu^{-1}_*(Q_i)\subset T$ its strict
transform. Note that both $Q_i$ and $D_i$ are smooth quadrics and
$D_i\cap D_j=f_k$, for $\{i,j,k\}=\{1,2,3\}$.

Let $p:Z\to T$ be the blow up of the $f_i$ with exceptional divisors
$E_i$. Then by construction I have $E_i\iso \P^1\times\P^1$,
$p^{-1}_*(D_i)\iso\P^1\times\P^1$, and in both cases the normal bundle
is $(0,-1)$. This shows that there is a birational morphism
$q:Z\to T_1$ that blows down the $D_i$'s. 
\begin{claim}
  $T_1\iso
T$ and $S_{T_1}:=(q\circ p^{-1})(S_T)\sim \pi_1^*\o(1)$.
\end{claim}
\begin{proof}[Proof of the claim] The varieties $T$, $Z$ and $T_1$ are
  all scrolls. $T$ and $T_1$ over $\P^2$, and $Z$ over the blow up of
  $\P^2$ in 3 non collinear points, say $W$. Let $T_1=\P({\mathcal
    E}_1)$ and $\eta:W\to\P^2$ and $\xi:W\to\P^2$ be the morphisms at
  the base level, induced by $p$ and $q$ respectively,
$$\xymatrix{
&\P({\mathcal E}_Z)=Z\ar[dl]_{p}\ar[dr]^{q}\ar[dd]_{\pi_Z}&\\
\P({\mathcal E})=T\ar[d]_{\pi}&&T_1=\P({\mathcal E_1})\ar[d]_{\pi_1}\\
\P^2&W\ar[l]_{\eta}\ar[r]^{\xi}&\P^2,}
$$
In particular I have
$\eta^*{\mathcal E}\cong {\mathcal E}_Z\cong\xi^*{\mathcal E}_1$.
 The map $\xi\circ\eta^{-1}$ is a standard Cremona modification, thus  $\eta$ and
 $\xi$ are both the blow blow-ups of three non collinear points. Then
 I have ${\mathcal E}={\mathcal
   E}_1$ up to a twist. In particular $T\iso T_1$. Moreover, the
 choice of $f_i\subset S_T$ yields $(q\circ p^{-1})(S_T)\sim \pi_1^*\o(1)$.
\end{proof}

By the Claim $T_1\iso T$ and $S_{T_1}\sim \pi_1^*(\o(1))$.  By
construction the 3-fold
$T$ had the contraction $\nu:T\to\P^3$. Therefore there is a contraction
$\nu_1:T_1\to \P^3$ that sends $S_{T_1}$ to a quadric, This shows that $S$ is CE to a quadric and therefore
it  is CE to a plane by Proposition~\ref{rem:cubics}.
\end{proof}
\begin{remark}
 
I want to give an alternative way to interpret the birational
modification described in the proof of Proposition~\ref{exa:twistedcubic}.
Consider three secant lines
$f_1$, $f_2$, $f_3$ to $\Gamma$ and the linear system $\Lambda$ of cubics
containing $R:=\Gamma\cup l_1\cup l_2\cup l_3$. It is easy to see that
$\dim\Lambda=3$. Observe that a general
element $D\in\Lambda$ is a smooth cubic. In particular $D$ is
isomorphic to a plane blow up
in 6 general points, say $\{q_1,\ldots,q_6\}$.  The reducible curve
$R$ is represented, in the
plane, by the  $4$ conics passing through a subset of 
$\{q_1,\ldots,q_6\}$ in such a way that any point is triple for
$R$. Therefore  the plane model of $\Lambda_{|D}$ is a linear system
of dimension 2,  degree 9, multiplicity $3$ in $\{q_1,\ldots,q_6\}$,
and  containing the degree 8 curve $R$.
This shows that  the plane model of $\Lambda_{|D}$ is $R+\o(1)$ and
therefore it  induces a birational map onto $\P^2$. That is  $\Lambda$
induces a birational map onto $\P^3$.

This modification is a degeneration
of the classical cubo cubic Cremona modification centered on a curve
of degree 6 and genus 3 of $\P^3$.
\end{remark}




The following Theorem settles the CE problem for rational quartic surfaces.
\begin{proposition}
  \label{cor:CE4} Let $S\subset\P^3$ be a rational surface of degree
  $4$, then $S$ is CE to a plane.
\end{proposition}
\begin{proof}  If $S$ is a cone then its general plane section is
  rational and I conclude by  \cite{Me2}. If $S$ has a
  point of multiplicity $3$, I conclude by Remark~\ref{rem:monoids}.

  From now on I assume that $S$ has only singular points of
  multiplicity $2$.
  The first step in the proof is to show that  rational quartic surfaces with
  isolated double points are CE to quartic surfaces with non isolated
  singularities.
  
 Assume that $S$ has isolated singularities. Then the rationality of
 $S$ forces the presence of an elliptic point. By a computation on
 Leray spectral series, see \cite{Um},  $S$ has a unique irrational
 singularity. Furthermore, by \cite{Je16} and \cite{De} classification, the irrational
 singularity is of the following type, in brackets the corresponding equation of $S$:
 \begin{itemize}
 \item[(1)] a double point with an
   infinitely near double line

   [$x_0^2x_1^2+x_0x_1Q_2(x_2,x_3)+F_4(x_1,x_2,x_3)=0$], 
 \item[(2)]  a tachnode with an infinitely near
   double line

   [$x_0^2x_1^2+x_0(x_2^3+x_1Q_2(x_2,x_3))+F_4(x_1,x_2,x_3)=0$]. 
\end{itemize}

Let $S$ be a rational quartic with a singular point of type $(a)$, $a\in\{1,2\}$, and  let $\Lambda_a\subset|\O(2)|$ be the linear
system of quadrics having multiplicity  $a+1$  on the
valuation associated to the double line. Then it is easy to check that
the map  $\f_{\Lambda_a}:\P^3\rat
X_a\subset\P^{7-a}$ is birational.

As observed in \cite[Example 4.3]{MP2} $X_1\subset\P^6$ is the cone over
the Veronese surface. With a similar argument, 
$X_2\subset\P^5$ is the projection of $X_1$ from any  smooth
point $z\in X_1$. This shows that  $X_2$ is the cone over 
the cubic surface $C\subset\P^4$, where $C$ is the projection of the
Veronese surface, say $V$,  from a point $z\in V$.

The main point here is that in both cases I have $S_a:=\f_{\Lambda_a}(S)\subset|\O_{\P^{7-a}}(2)|$.
\begin{case}[$S_2$] Assume thst $S$ has a point of type (2). 
Then the pair $(\P^3,S)$ is birational to $(X_2,S_2)$.
The surface $S_2\subset X_2\subset\P^5$ has degree 6 and $X_2$
has degree $3$. Let $x\in S_2$ be a general point and
$\pi:\P^5\rat\P^4$ the projection from $x$. Then $\pi(X_2)=Q$ is a
quadric cone and $S_x:=\pi(S_2)$ is a surface of degree 5. Hence there
is a a cubic hypersurface $D\subset\P^4$ such that
$$D_{|Q}=S_x+H,$$ for some plane $H$. Let $y\in S_x$ be a general
point and $\pi_y:\P^4\rat\P^3$ the projection from $y$.
\begin{claim}
  $\tilde{S}:=\pi_y(S_x)$ is a quartic surface singular along a line.
\end{claim}
\begin{proof}
  The point $y$ is general therefore $\deg\tilde{S}=4$. The map
  $\pi_{y|Q}$ is birational and it contracts  the embedded tangent cone
  ${\mathbb T}_yQ\cap Q=\Pi_1\cup\Pi_2$ to a pair of lines $l_1\cup l_2$. Up to reordering I
  may assume that $H\cap \Pi_1$ is the vertex of the cone.  Therefore $\tilde{S}\cap
  \Pi_1$ is a cubic passing through $y$. Hence $\tilde{S}$ has
  multiplicity 2 along $l_1$.
\end{proof}

The surface  $S$ is therefore CE to a quartic with non isolated singularities.
\end{case}

\begin{case}[$S_1$] Assume that $S$ has a point of type (1). Then
  $(\P^3, S)$ is birational to $(X_1,S_1)$.
  \begin{claim} $S_1$ is in the smooth locus of $X_1$ and $S_1$ has
   at most  ordinary double points.
  \end{claim}
  \begin{proof} Let me start describing the map
    $\f:=\f_{\Lambda_1}:\P^3\rat X_1$, following \cite[Example
    4.3]{MP2},

    Let $S\subset\P^3$ be
  the quartic, I
  may assume that the irrational singular point is
  $p\equiv[1,0,0,0]\in S$ and the equation of $S$ is
$$(x_0^2x_1^2+x_0x_1Q+F_4=0)\subset\P^3.$$
Let
   $\epsilon:Y\to\P^3$ be the weighted
blow up of $p$,
  with weights $(2,1,1)$ on the
  coordinates $(x_1,x_2,x_3)$, and
  exceptional divisor $E\iso\P(1,1,2)$. Then I have:
  \begin{itemize}
  \item[-] $\epsilon^*(x_1=0)=H+2E$,
  $\epsilon_{|H}:H\to (x_1=0)$ is an
  ordinary blow up and $H_{|E}$ is a smooth rational curve;
  \item[-]  $\epsilon^*(S)=S_Y+4E$, $S_{Y|E}$  is a
   smooth
  elliptic curve, and $S_{Y|H}$ is a union of four smooth disjoint
  rational curves.
  \end{itemize}
   In particular both
  $H$ and $S_Y$ are on the smooth locus of $E$ and hence on the smooth
  locus of $Y$. The surface $H$ is ruled by, the strict transforms of, the
  lines in the plane $(x_1=0)$ passing through the point $p$. Let
  $l_Y$ be a general curve in the ruling and
  $\Lambda_Y=\epsilon^{-1}_*(\Lambda_1)$ the
  strict transform linear system. Then  $E\cdot l_Y=1$ and by a direct computation I have
  \begin{itemize}
  \item[-] $\Lambda_Y\cdot
    l_Y=(\epsilon^*(\O(2))-2E)\cdot l_Y=0$
    \item[-] $S_Y\cdot
    l_Y=(\epsilon^*(\O(4))-4E)\cdot l_Y=0$
  \item[-] $K_Y\cdot l_Y=(\epsilon^*(\O(-4))+3E)\cdot l_Y=-1$.
  \item[-] $H\cdot l_Y=(\epsilon^*(\O(1))-2E)\cdot l_Y=-1$.
  \end{itemize}
  Then $H$  can be blown down to a smooth rational
  curve with a birational map $\mu:Y\to
  X_1$ and by construction $S_Y=\mu^*S_1$.  This shows that  the unique singularity of
  $X_1$ is the singular point in $E$ and $S_Y$ has at most isolated
  rational double points.
\end{proof}
If $S_1$ is smooth then $(X_1,S_1)$ is a good model of $(\P^3,S)$ with
threshold $\rho(X
_1,S)=4/5$ and I conclude by  Corollary~\ref{cor:rhominorediuno}.
Assume that $S_1$ is singular and let $x\in\Sing(S_1)$  be a singular
point. Set $\pi:\P^6\rat\P^5$ be the projection from $
x$. Then $\pi_{|X_1}:X_1\rat X_2$ is birational and $\pi(S_1)\in
|\O_{X_2}(2)|$. I am therefore back to the previous case. This shows
that $(\P^3,S)$ is CE to a quartic with non isolated singularities.
\end{case}

To conclude the Proposition I am  
left to study the case of quartics with non isolated
singularities.

From now on I fix  a rational quartic $S$ with a curve $\Gamma$ of double
points.
Assume first that $\Gamma$ contains a line $l$. Fix a general point
$x\in S$ and the linear system  $\Lambda$ of
quadrics through $l$ and $x$. Let $\f:\P^3\rat\P^5$ be the map
associated to the linear system $\Lambda$. I have
$\f(\P^3)=Z\iso\P^1\times\P^2$, embedded via the Segre map, and $\f(S)=\tilde{S}$ is a divisor of type
$(3,2)$ in $\P^1\times\P^2$. Note that divisors of type $(1,0)$ are
planes and divisors of type $(0,1)$ are quadrics, then I have $\deg\tilde{S}=3+4=7$.
If $\tilde{S}$ is smooth then $(Z,\tilde{S})$  is a good
model of $(\P^3,S)$ with $\rho(Z,\tilde{S})=2/3$ and I conclude by Corollary~\ref{cor:rhominorediuno}.
If $\tilde{S}$ is singular let $y\in\Sing(\tilde{S})$ be  a point and $\pi:\P^5\rat\P^4$ the
projection from $y$. Then $\pi_{|Z}$ is a birational map, $Y:=\pi(Z)\subset\P^4$
   is a quadric of rank $4$, and $S_Q:=\pi(\tilde{S})$ is a
  rational surface of degree $7-2=5$.
  \begin{claim}
    The vertex of the quadric is a smooth point  of $S_Q$.
  \end{claim}
  \begin{proof}
The surface    $\tilde{S}$ is a divisor of type $(3,2)$ in $Z$ and
    it is singular in $y$. Let $l$ and $P$, respectively, be the line
    and the plane passing through $x$ in $Z$. The general choice of
    $x\in S$ yields $l\not\subset \tilde{S}$. The line $l$ is
    mapped to the vertex of the quadric and $\tilde{S}_{|l}=2x+p$ for some point
    $p$. This shows that $S_Q$ contains the vertex of the quadric and
    it is smooth there.
  \end{proof}

  The 3-fold $Q$ is a quadric cone. If $S_Q$  is smooth let $\nu:T\to Q$ be a $\Q$-factorialization of
 $Q$ and $S_T$ the strict transform of $S_Q$. Then $(T,S_T)$ is a good
 model for $(\P^3,S)$ and $\rho(T,S_T)=2/3$. Therefore I conclude by Corollary~\ref{cor:rhominorediuno}. 
If $S_Q$ is singular let $z\in \Sing(S_Q)$ be a point. By the Claim it is not
the vertex of $Q$. Thus the projection from $z$ produces a birational
model of $(Q, S_Q)$ , say $(\P^3,Z)$, with  $Z$ a rational cubic and I
conclude  by Lemma~\ref{rem:cubics}.

Assume that $\Gamma$ does not contain a line.
It is easy to see that $\deg\Gamma\leq 3$. Moreover  if $\deg\Gamma=3$
the curve $\Gamma$ is a
twisted cubic. Therefore I am  left to
consider the following cases: $\Gamma$ an irreducible conic, $\Gamma$
a twisted cubic.
If $\Gamma$ is a conic let $x\in S$ be a general point. Then the linear system of quadrics
through $\Gamma$ and $x$ maps $(\P^3,S)$ to the pair $(\P^3,S^\prime)$ with $S^\prime$ a rational cubic surface, and I conclude
again by  Lemma~\ref{rem:cubics}.
If $\Gamma$ is a twisted cubic I conclude by Proposition~\ref{exa:twistedcubic}.
\end{proof}

\begin{example}
  I finish giving an explicit example of linearization of quartics
  that has been classically studied for being
  envelopes of bitangent spheres, \cite[Chapter V]{Je16}: the
  Cyclides.
  I thank Igor Dolgachev for pointing out this special class of
  quartics and Alex Massarenti for working out the explicit equations
  with Macauley2

  Let $S\subset\P^3$ be a quartic with the following equation
  $$(x^2+y^2+z^2-w^2)^2+w^2q=0, $$
  where $q$ is a polynomial of degree $2$.
The surface $S$ is singular along the conic
$C=(w=x^2+y^2+y^2=0)$.
Assume first that there is a further singularity, say $p\not\in C$,
this is for instance the case of the Dupin's cyclid.
 Let $Q$ be a general quadric
through $C\cup p$. Then $Q\cap S=2C+R$ for some residual curve
$R$. The residual curve $R$ is a quartic rational curve.
Fix a general point in $q\in S$. Then the linear system that
linearizes $S$ is
$$\Lambda=|\I_{C^2\cup p^2\cup R\cup q}(4)|.$$
If $p$ is defined over the field of real numbers then the map is defined over the
real numbers. Here is a sample with plausible equations. Let us start
with
$$S=((x^2+y^2+z^2-w^2)^2+w^2((w-x)^2+y^2-z^2)=0)$$
the singular point is $p=[1,0,0,1]$ and the general point
$q=[0,0,1,1]$. Then a linear system that linearize $S$ is
\begin{eqnarray*}
  \{x^2 zw + y^2zw + z^3w + 2x^2w^2  + y^2w^2  + xzw^2  - z^2w^2  -
  4xw^3  - 2zw^3  + 2w^2 ,&\\
  x ^2yw + y^3w +yz^2w + xyw^2  - 2yw^3,
  x^3w + xy^2w + xz^2w - 2x^2w^2  - 2y^2w^2  + xw^3,&\\
    x^4+ 2x^2y^2  + y^4  + 2x^2z^2 +2y^2 z^2  + z^4  - x^2 w^2  - y^2 w^2  - 3z^2w^2  - 2xw^3  + 2w^4\}.
\end{eqnarray*}
 
If $S$ has not further double points let $\Gamma$ be a smooth rational quartic curve in $S$. 
Let $p\in \Gamma\in S$ be a general point and consider the linear system
$$\Sigma=|\I_{C^3\cup p^3\cup\Gamma}(6)|.$$
Let $D\in\Sigma$ be a general element then
$$ D\cap S=6C+\Gamma+ R$$
This time the residual curve $R$ has degree 8 and genus 3. The linear system that linearizes $S$ is then
$$\Lambda=|\I_{C^3\cup p^3\cup R}(6)|.$$
Denote by $E$ the cone over $C$ with vertex $p$, then $E+S\subset\Lambda$ and $E$ is contracted by $\f_\Lambda$.

It is not clear to me if any such cyclid contains a rational quartic
curve defined over the field of real numbers. 
\end{example}


\end{document}